\documentclass[12pt]{amsart}
\advance\hoffset by -0.5 truecm
\usepackage{amssymb}
\newtheorem{Theorem}{Theorem}[section]
\newtheorem{Lemma}[Theorem]{Lemma}

\newtheorem{Proposition}[Theorem]{Proposition}

\usepackage{enumerate}
\usepackage{graphicx}
\usepackage{subfigure}
\usepackage{verbatim}

\def\V{\mbox{Var}}

\def\R\re
\def\V{\bf V}

\def \re{{\mathbb R}}

\def \0{\lambda_{0}}

\def \S{{\bf S}}

\begin{document}
\title{Minimal hypersurfaces in $\re^n \times {\bf S}^m$}

\author[J. Petean]{Jimmy Petean}
 \address{CIMAT  \\
          A.P. 402, 36000 \\
          Guanajuato. Gto. \\
          M\'exico.}
\email{jimmy@cimat.mx}

\author[J. M. Ruiz]{Juan Miguel Ruiz}
 \address{ENES Unidad Le\'on - UNAM  \\
           Blvd.UNAM 2011 \\
           C.P. 37684 \\
          Le\'on Gto. \\
          M\'exico.}
\email{mruiz@enes.unam.mx}




\begin{abstract} We classify minimal hypersurfaces in $\re^n \times {\bf S}^m$, $n,m \geq 2$,  which are invariant by the
canonical action of $O(n) \times O(m)$. We also construct compact and noncompact examples of invariant 
hypersurfaces of constant mean curvature.
We show that the minimal hypersurfaces and the noncompact constant mean curvature hypersurfaces are all unstable.

\end{abstract}

\maketitle

\section{Introduction}
In  this article we will  construct families of complete embedded minimal hypersurfaces in the Riemannian products
$\re^n \times{\bf S}^m$,  and also some examples of families of complete embedded hypersurfaces of constant mean curvature. 
The study of minimal hypersurfaces is a very classical problem in differential geometry and there is a large literature on construction of examples. 
The most studied cases are the minimal surfaces in ${\bf S}^3$ and $\re^3$. Let us mention  for instance that 
H. B. Lawson proved in \cite{Lawson} that every compact orientable surface can be embedded as a minimal
surface in ${\bf S}^3$. Further constructions of minimal surfaces in the sphere were done by H. Karcher, U. Pinkall and
I. Sterling in \cite{Karcher} and by N. Kapouleas and S. D. Yang in \cite{Kapouleas}, among others. 
There are also plenty of constructions in other 3-manifolds, for instance recently 
F. Torralbo  in \cite{Torralbo} builded examples of minimal surfaces immersed in the Berger spheres. 
And in the last few years there has been great interest in the study of minimal surfaces in 3-dimensional 
Riemannian products: see for instance the articles by H. Rosenberg \cite{Rosenberg}, W.H. Meeks and 
H. Rosenberg \cite{Meeks} and J. M. Manzano, J. Plehnert and F. Torralbo \cite{Manzano}. 
In higher dimensional manifolds constructions are also abundant. Much work has been done considering
minimal hypersurfaces which are invariant by large groups of isometries. Very closely related to this work
is the article by H. Alencar \cite{Alencar} where the author studies minimal hypersurfaces in  $\re^{2m}$ invariant by the
action of $SO(m) \times SO(m)$. The general case of $SO(m)\times SO(n)$-invariant minimal hypesurfaces in
$\re^{m+n}$ was later treated  by H. Alencar, A. Barros, O. Palmas, J. G. Reyes and W. Santos in \cite{Alencar2}.  They
give a complete description of such minimal hypersurfaces when $m,n \geq 3$. Of great interest for the present work is also
the article by R. Pedrosa and M. Ritor\'{e} \cite{Pedrosa} where the authors study the isoperimetric problem in the
Riemannian products of $n$-dimensional simply connected spaces of constant curvature and circles. The isoperimetric
regions in these spaces are known to exist (see the articles by F. Almgren \cite{Almgren} and  F. Morgan \cite{Morgan1, Morgan2}) and
their boundaries are hypersurfaces of constant mean curvature which are invariant by the action of the orthogonal group acting
on the space with constant curvature (fixing a point). Therefore in the study of the isoperimetric problem in such regions one is
naturally led to study invariant hypersurfaces of constant mean curvature.  A detailed study of such hypersurfaces, in particular
minimal ones, was carried out by R. Pedrosa and M. Ritor\'{e} in their article. 
When studying minimal hypersurfaces invariant by the group actions as in the  articles mentioned above one is essentially dealing with
the solutions of an ordinary differential equation. The same is true for the case of hypersurfaces in   $\re^n \times {\bf S}^m$ which are invariant by the canonical
action of $O(n) \times O(m)$, which is the situation we will study in the present article. The main technical difference when changing the
circle for higher dimensional spheres is that in the first case due to the invariance of the problem by rotations of the circle the associated 
ordinary differential equation has a first integral, which helps to describe the solutions of it: this is not true when $n, m \geq 2$. 

The first and main goal of this article is to classify minimal hypersurfaces in $\re^n \times {\bf S}^m$ invariant by the
canonical action of $O(n) \times O(m)$.
Of course there is one canonical such minimal hypersuface given
by the product of $\re^n$ with a maximal hypersphere  $ {\bf S}^{m-1} \subset {\bf S}^m$. The orbit space of the action of
$O(n) \times O(m)$ in $\re^n \times {\bf S}^m$ is identified with $[0,\infty ) \times [0, \pi ]$. And an invariant hypersurface is
described by a {\it generating curve} $\varphi $ in the orbit space. The corresponding hypersurface has mean curvature 
$h$ if the curve $\varphi$ satisfies 

$$\sigma'(s) = (m-1) \ \frac{\cos (y(s))}{\sin (y(s))} \cos (\sigma (s)) - (n-1) \ \frac{\sin \sigma (s) }{x(s) } - h(s) .$$

\noindent
where $\varphi = (x(s), y(s))$ is parametrized by arc length and $\varphi ' (s) = \cos (\sigma (s) , \sin \sigma (s) )$
(and $h$ is taken  with respect to the normal vector ${\bf n}=(\sin \sigma (s) , -\cos \sigma (s))$. 
We will describe the invariant minimal (or constant mean curvature) hypersurfaces by studying solutions to this equation with
$h=0$ (or a nonzero constant).

We will give a complete description of invariant minimal hypersurfaces:

\begin{Theorem} Let $n,m \geq 2$ and consider hypersurfaces  in   $\re^n \times {\bf S}^m$ invariant by the action of $O(n) \times O(m)$.
There is a one dimensional family of invariant minimal  embeddings of $\re^n \times {\bf S}^{m-1}$ parametrized by $r\in (0,\pi )$.
There is a 2-dimensional family of invariant  minimal immersions (with self-intersections)
 of ${\bf S}^{n-1} \times {\bf S}^{m-1} \times \re$ parametrized by
$(r,s) \in (0,\infty ) \times (0,\pi )$. There are two 1-dimensional families of invariant minimal embeddings of ${\bf S}^{n-1} \times \re^m$
parametrized by $r\in \re_{>0}$ (each family is the reflection of the other around $\re^n \times {\bf S}^{m-1}$).
These are all minimal hypersurfaces in $\re^n \times {\bf S}^m$  invariant by the action. 
\end{Theorem}

Figure 1 shows examples of the corresponding generating curve for each of the three cases (they correspond to the case $m=n=2$).

\begin{figure}[h!]

\subfigure[$ (x_0,y_0,\sigma_0)=(0,3,0)$]{
	 		 			  \includegraphics[scale=0.500]{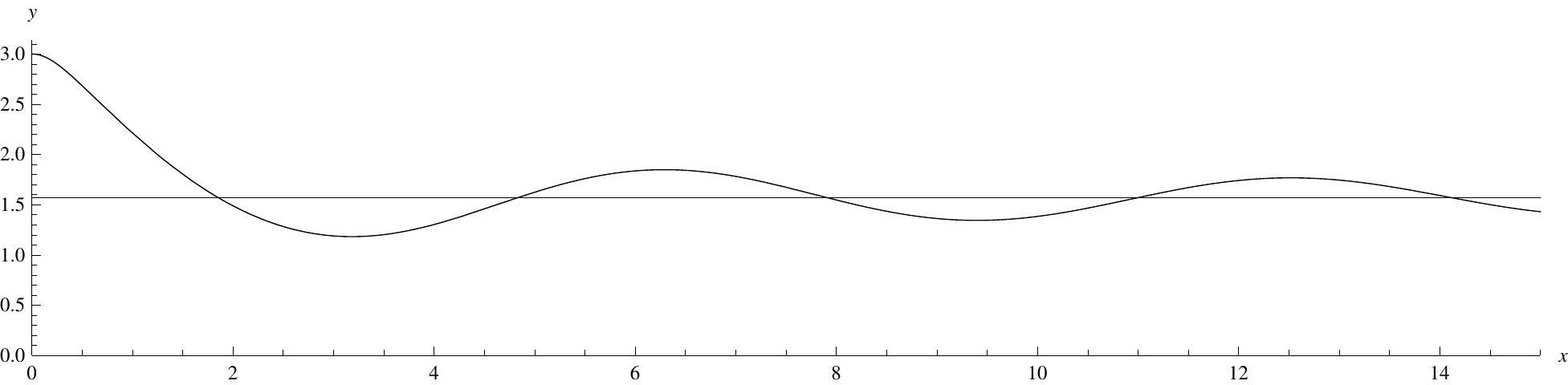}}
\subfigure[$ (x_0,y_0,\sigma_0)=(1,0,\frac{\pi}{2}))$.]{
 	 		 			  \includegraphics[scale=0.500]{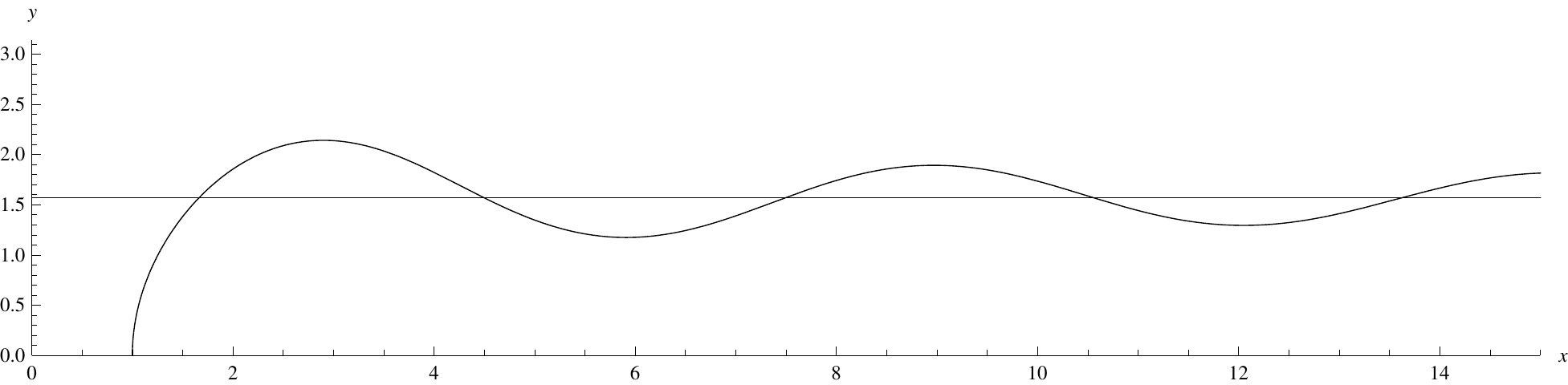}} 		 	

\subfigure[$(x_0,y_0,\sigma_0)=(1,2,\frac{\pi}{2})$.]{		
 	 		 			  \includegraphics[scale=0.500]{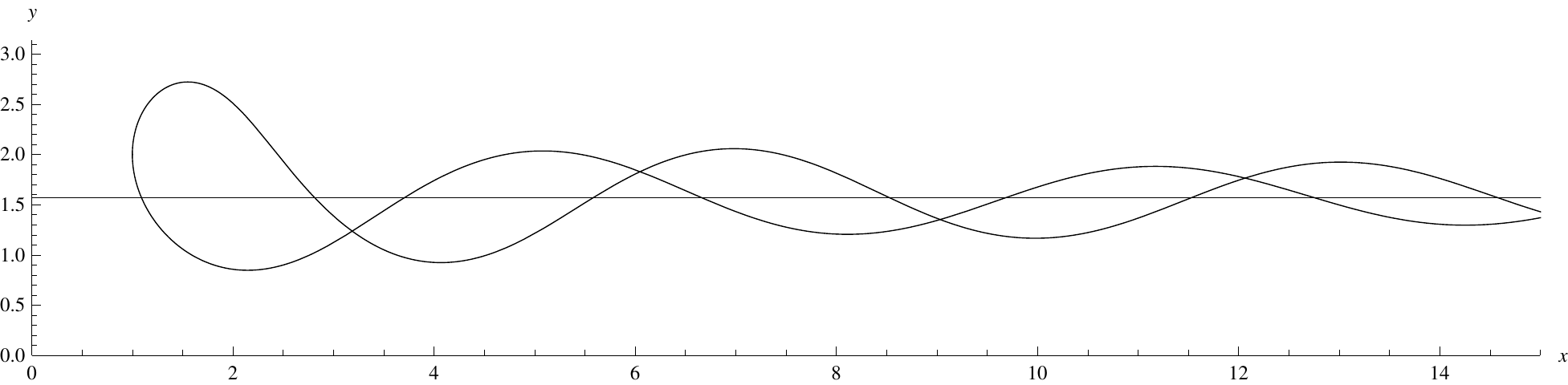}}
\subfigure[$ (x_0,y_0,\sigma_0)=(1,\frac{\pi}{2},\frac{\pi}{2})$.]{
 	 		 			  \includegraphics[scale=0.500]{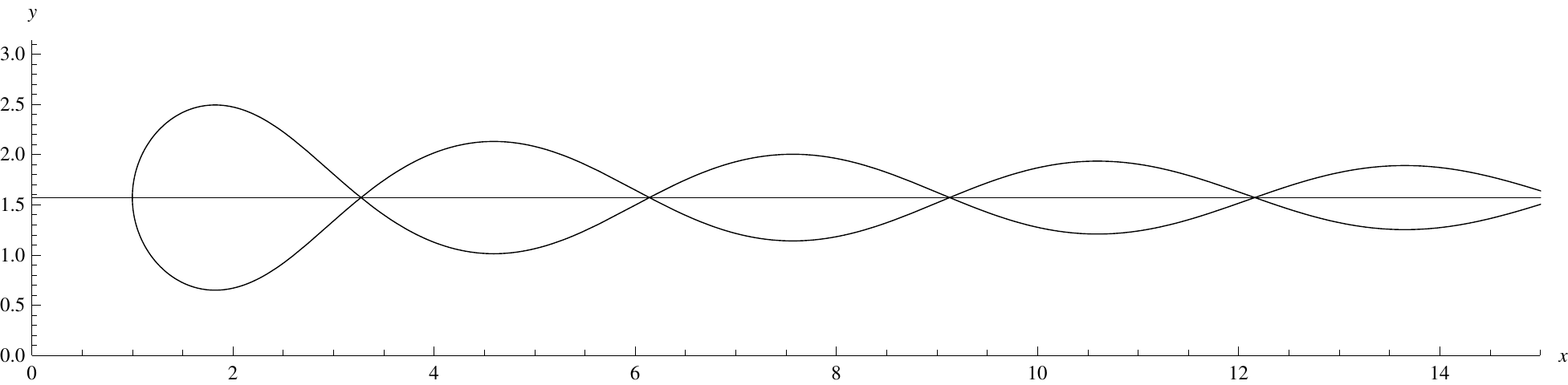}}

\caption{Generating curves of minimal hypersurfaces in $\S^2 \times \re^2$, with  initial conditions $ (x_0,y_0,\sigma_0)$.}
\label{fig:1}
\end{figure}

We will use similar techniques to show the existence of families of embedded constant mean curvature
hypersurfaces invariant by the action of $O(n) \times O(m)$. In this case the situation is considerably more
complicated and we will not give a complete description. One of the main reasons why one is interested in understanding invariant
hypersurfaces of constant mean curvature is to compute the isoperimetric function of  $\re^n \times {\bf S}^m$. By standard symmetrization arguments
one can see that an isoperimetric region in $\re^n \times {\bf S}^m$  can be of two types: either a product of the sphere with a ball in Euclidean space or a ball type region invariant 
by the action of $O(n) \times O(m)$. The boundary of a region of the second type is an invariant hypersphere of constant mean curvature: the corresponding
generating curve will start perpendicular to  the $y$-axis and decrease until reaching the $x$-axis.  We will then concentrate in hypersurfaces {\it starting} in the
$y$-axis, since these are the ones that could give isoperimetric regions. There is of course a canonical example: if $x_h \in (0,\pi )$ is defined by the
equation $\cot (x_h )= \frac{h}{m-1}$ and ${\bf S}^{m-1}_r$ is  the hypersphere of points at distance $r$ from the south pole $S\in {\bf S}^m$ then
$ \re^n \times  {\bf S}^{m-1}_{x_h}$ is an invariant hypersurface of constant mean curvature $h$. 
We will prove:

\begin{Theorem} For any $h \in \re_{>0}$ there is a one-dimensional family of $O(n) \times O(m)$-invariant embedded hypersurfaces of constant mean curvature
$h$ in $\re^n \times S^m$ diffeomorphic to $\re^n \times {\bf S}^{m-1}$ parametrized by $A\in (0,b)$ where $b \in (x_h , \pi ]$.
If  $b \neq \pi$ then there is an embedding of ${\bf S}^{n+m-1}$ with constant
mean curvature $h$, invariant by the action of  $O(n) \times O(m)$. 

\end{Theorem}

The isoperimetric problem   in   $\re^n \times {\bf S}^m$ can be solved. Moreover it is known that 
for small values of the volume the corresponding isoperimetric region must be a ball: explicitly we will point out in Lemma 5.2 that ${\bf S}^{n-1} (r) \times
{\bf S}^m \subset \re^n \times {\bf S}^m$ is unstable (as a constant mean curvature hypersurface) if $r<\sqrt{n-1}/\sqrt{m}$ (by
${\bf S}^{n-1} (r)$ we denote the sphere of radius $r$) . For instance in the 
case of $\re^2 \times {\bf S}^2$  this says that isoperimetric regions of volume 
less than $2\pi^2$ in  $ \re^2 \times {\bf S}^2$  are balls.

Therefore we know that for some (large) values of $h$ there are invariant hyperspheres
like in the second part of the theorem. Solving the equation numerically it seems that if  for some value of $h$ there exists such hypersphere 
then it is unique and it {\it divides} generating curves
 like the ones in the Theorem and generating curves with self-intersections. If one could prove that this is actually the case then one would have a good
understanding of the isoperimetric profile of $\re^n \times {\bf S}^m$. This should be compared to the case of spherical cylinders $\re \times {\bf S}^m$ treated 
by R. Pedrosa in \cite{Pedrosa2}.

In Figure 2 we show  three types of  generating curves of hypersurfaces of constant mean curvature $h=1.8$ that appear for  the
case $n=m=2$. There is a value $y_0  \approx 1.592$ such that the generating curve of the hypersurface of constant mean curvature $h$ starting at
$(0, y_0 )$ ends perpendicular at the $x$-axis, giving an embedded ${\bf S}^3$: the curves starting at $y<y_0$ produce embeddings of $\re^2 \times {\bf S}^1$ and
the curves starting at $y>y_0$ produce immersions with self-intersections. In Figure 3 we still consider $n=m=2$ but $h=3$. Again there is a value 
 $y_0  \approx 0.98$ such that the generating curve of the hypersurface of constant mean curvature $h$ starting at
$(0, y_0 )$ ends perpendicular at the $x$-axis, giving an embedded ${\bf S}^3$:  the curves starting at $y<y_0$ produce embeddings of $\re^2 \times {\bf S}^1$ and
the curves starting at $y>y_0$ produce immersions or embeddings of constant mean curvature hypersurfaces, but for which there is a point with $y>0, \  x'=0$ (and
the corresponding hypersurface cannot be the boundary of an isoperimetric region).

 \begin{figure}[h!]
 
  \subfigure[$(x_0,y_0,\sigma_0)=(0,1,0)$.]{		
   	 		 			  \includegraphics[scale=0.750]{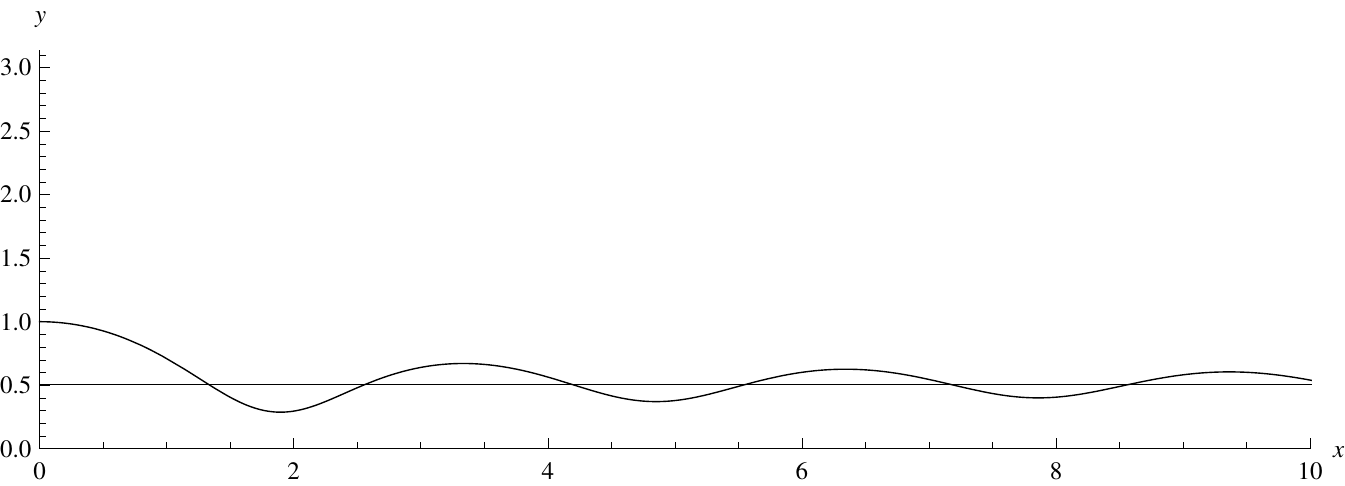}}
 \subfigure[$ (x_0,y_0,\sigma_0)=(0,1.592,0).$]{
 	 		 			  \includegraphics[scale=0.750]{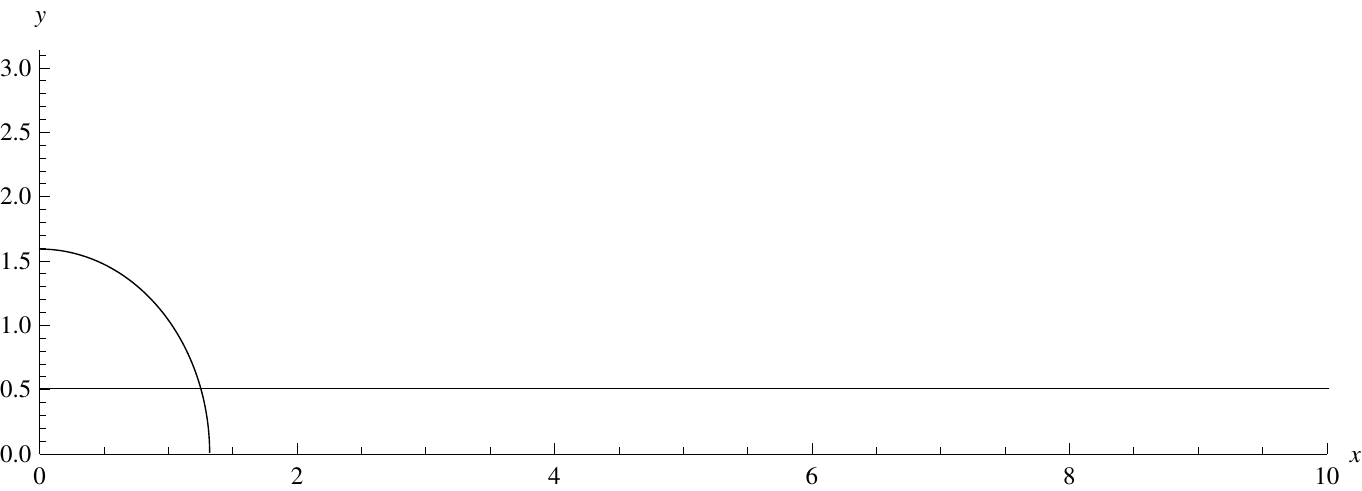}}
 \subfigure[$ (x_0,y_0,\sigma_0)=(0,2,0)$.]{
  	 		 			  \includegraphics[scale=0.750]{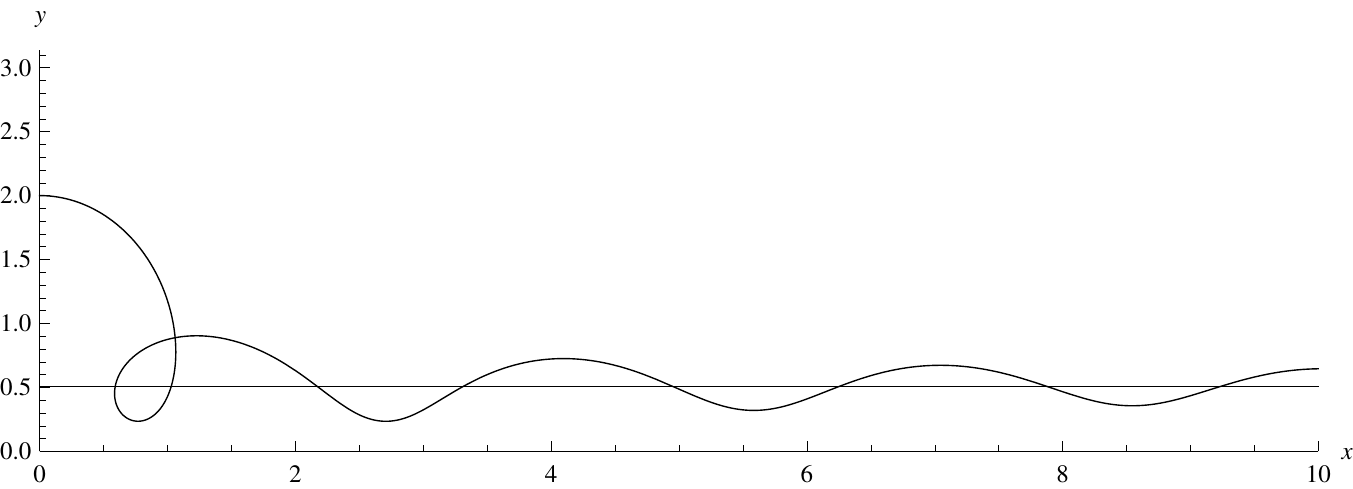}}

 \caption{Generating curves of minimal hypersurfaces in $\S^2 \times \re^2$, with  initial conditions $ (x_0,y_0,\sigma_0)$ and constant mean curvature $h=1.8$.}
 \label{fig:2}
 \end{figure}

  \begin{figure}[h!]
  
   \subfigure[$(x_0,y_0,\sigma_0)=(0,0.6,0)$.]{		
    	 		 			  \includegraphics[scale=0.50]{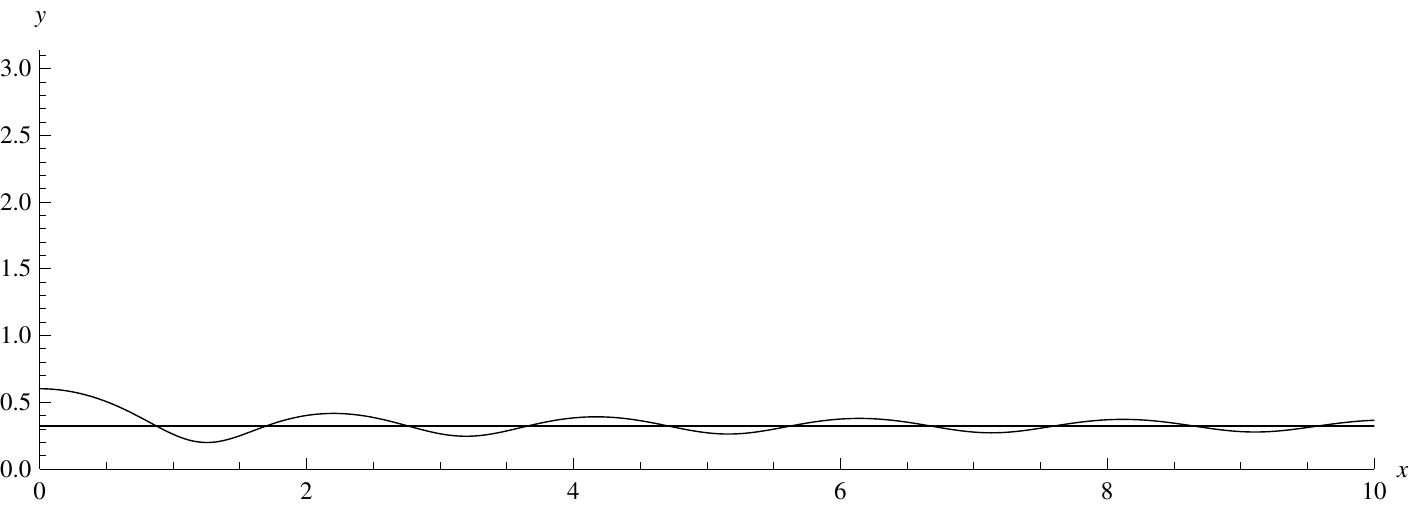}}
  \subfigure[$ (x_0,y_0,\sigma_0)=(0,0.98,0).$]{
  	 		 			  \includegraphics[scale=0.50]{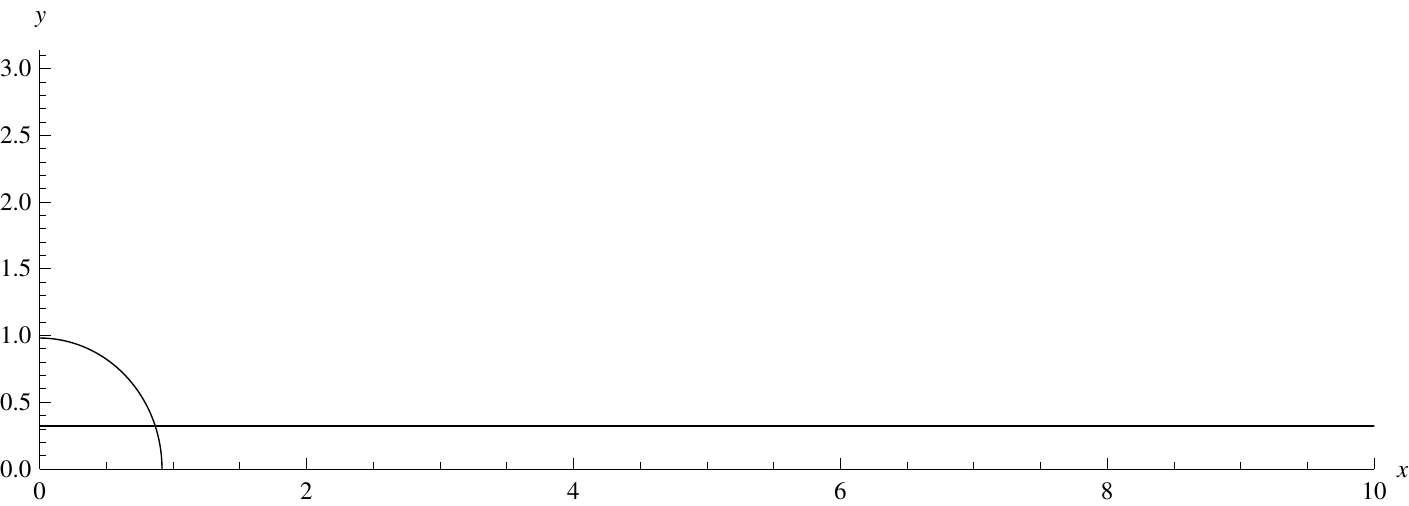}}
  \subfigure[$ (x_0,y_0,\sigma_0)=(0,1.3,0)$.]{
   	 		 			  \includegraphics[scale=0.50]{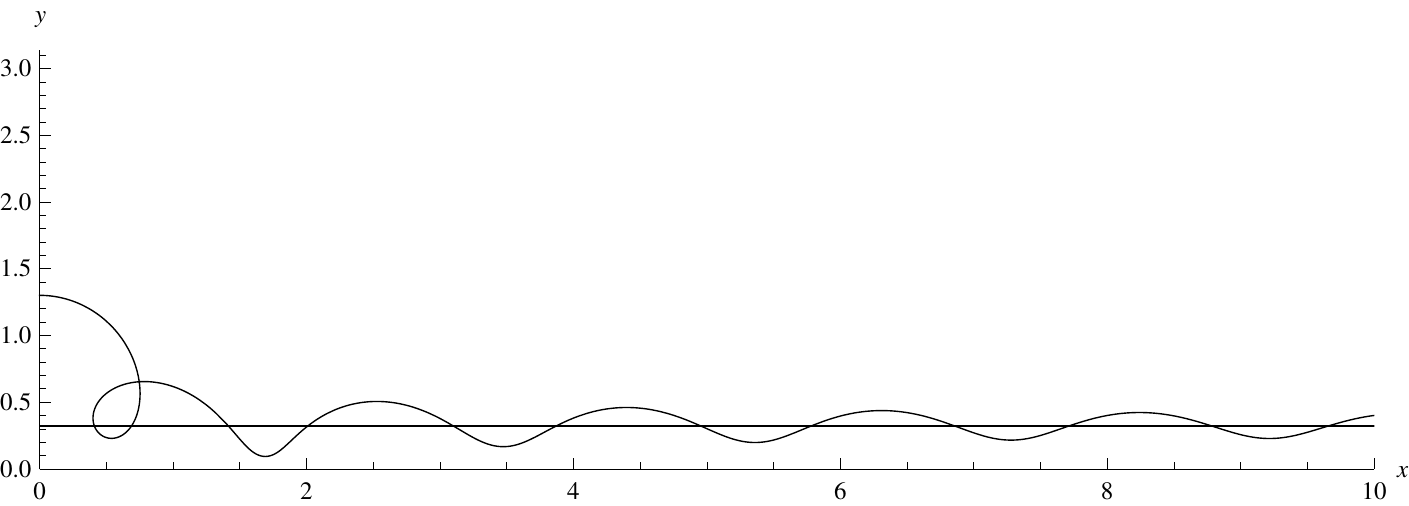}}      \subfigure[$(x_0,y_0,\sigma_0)=(0,1.68,0)$.]{		
      	 		 			  \includegraphics[scale=0.50]{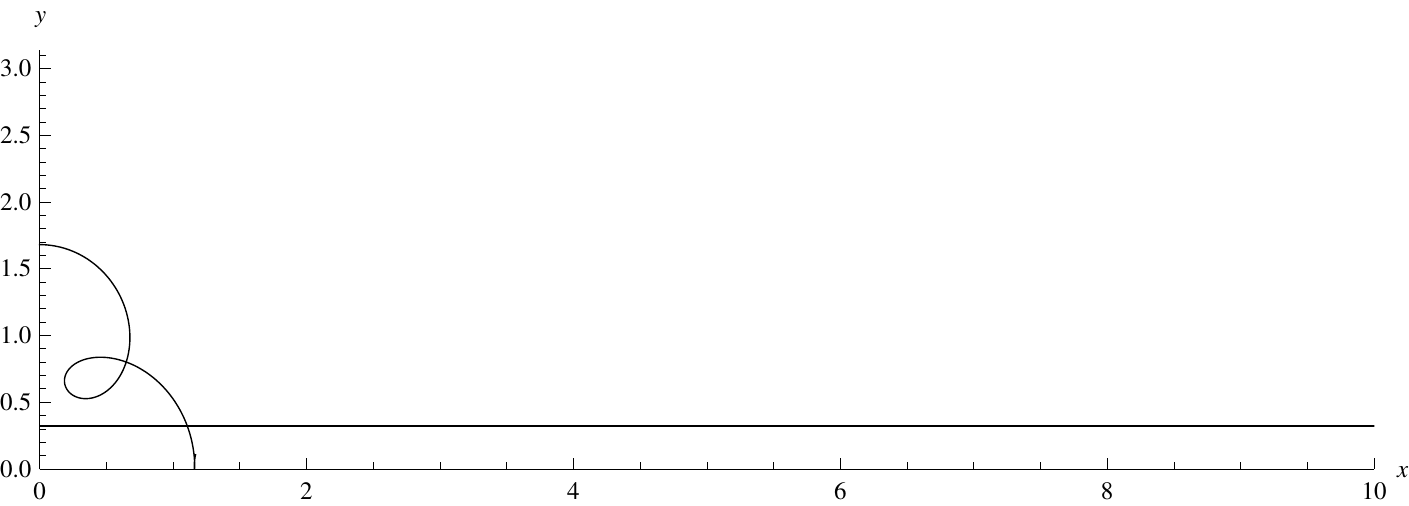}}
    \subfigure[$ (x_0,y_0,\sigma_0)=(0,2,0).$]{
    	 		 			  \includegraphics[scale=0.50]{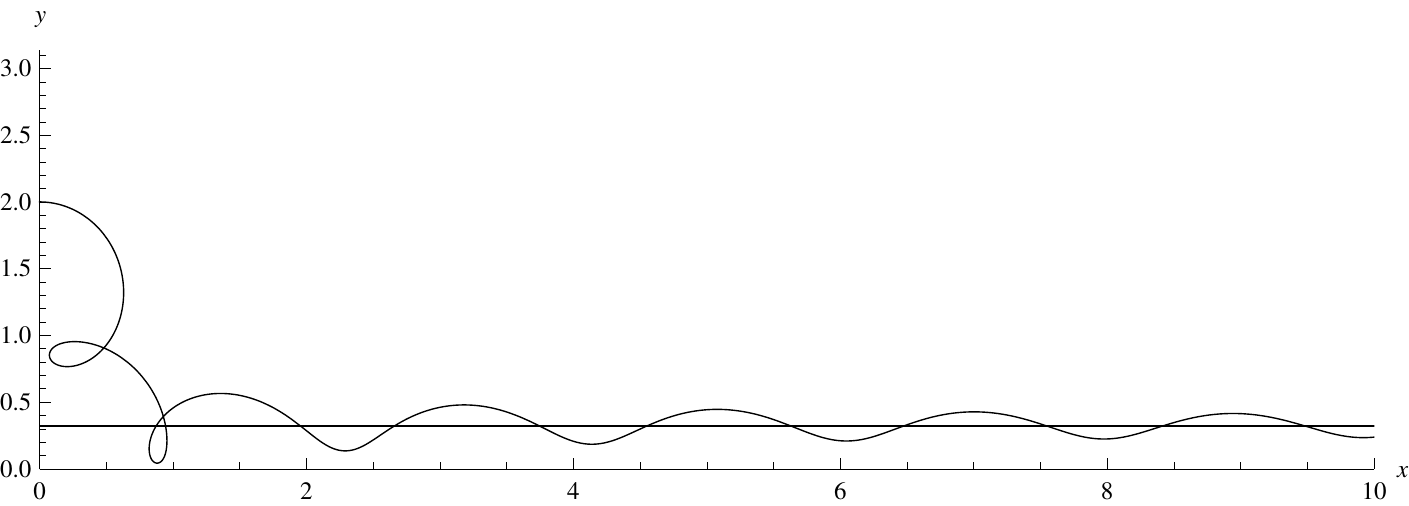}}
    \subfigure[$ (x_0,y_0,\sigma_0)\approx(0,2.24,0)$.]{
     	 		 			  \includegraphics[scale=0.50]{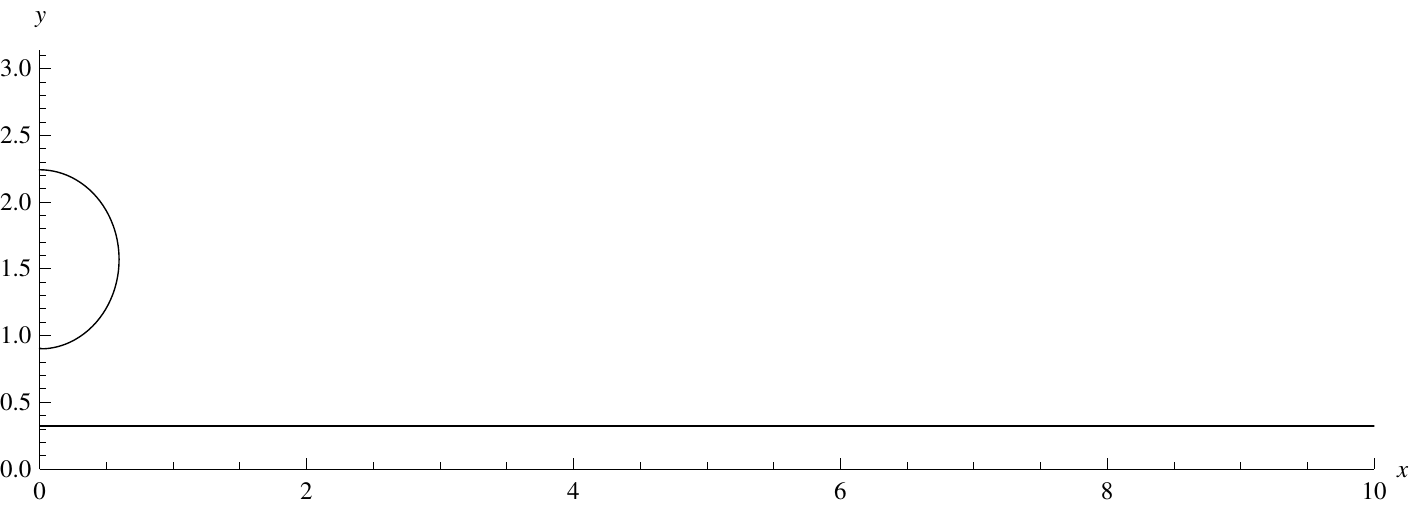}} 
     	 		 			     \subfigure[$(x_0,y_0,\sigma_0)=(0,2.8,0)$.]{		
     	 		 			   \includegraphics[scale=0.50]{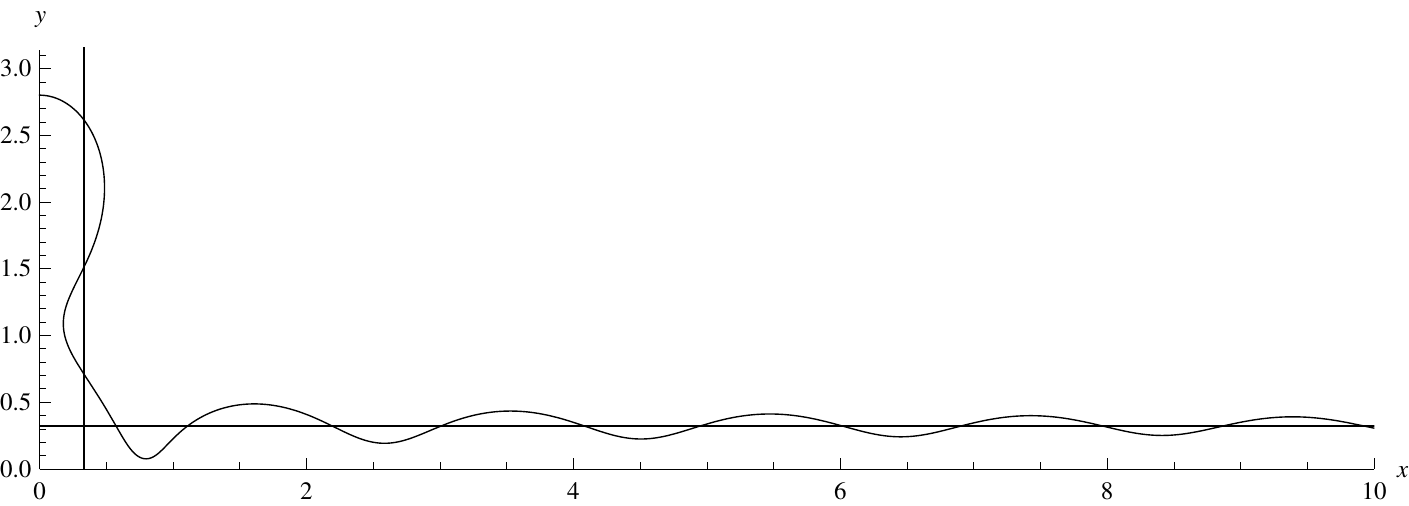}}

  \caption{Generating curves of minimal hypersurfaces in $\S^2 \times \re^2$, with  initial conditions $(x_0,y_0,\sigma_0)$ and constant mean curvature $h=3$.   
  }
  \label{fig:3}
  \end{figure}

Finally we will discuss the stability of the noncompact constant mean curvature  hypersurfaces discussed  in the previous theorems. We will prove: 

\begin{Theorem}
All minimal hypersurfaces in  $\re^n \times {\bf S}^m$ invariant by the action of $O(n) \times O(m)$ and all the
noncompact invariant constant mean curvature hypersurfaces constructed in Theorem 1.2 are unstable. 

\end{Theorem}


\section{$O(n) \times O(m)$-invariant  hypersurfaces}

We consider a hypersurface $M^{m+n-1} \subset  \re^n  \times  {\bf S}^m$ invariant by the canonical action of
$O(n) \times O(m)$ (fixing the south pole  in ${\bf S}^m$ and the origin in $\re^n$). The orbit
space of this action is identified with $  [0,\infty ) \times [0, \pi ]$.  $M$ is identified with a
curve $\varphi$  in $[0,\infty ) \times [0, \pi]$. Parametrizing this curve by arc length (and picking an orientation)
and writing $\varphi (s) = (x(s),y(s))$, then $x'(s) = \cos \sigma (s)$, $y' (s) = \sin \sigma (s)$ where 
$\sigma (s)$ is the angle formed by the (oriented) curve and the $x$-axis at  $\varphi (s) $. 

The mean curvature $h$ of $M$ is of course invariant by the action of
$O(n) \times O(m)$ and so it can be expressed as a function of the parameter $s$. It is given  in terms of the curve $\varphi$
and with respect to the normal vector ${\bf n}=(\sin \sigma (s) , -\cos \sigma (s))$, by

\begin{equation} h(s)= (m-1) \ \frac{\cos (y(s))}{\sin (y(s))} \cos (\sigma (s)) - (n-1) \ \frac{\sin \sigma (s) }{x(s) } - \sigma '(s) .
\end{equation}

Of course if one changes the orientation of $\varphi$ considering the curve 
$\overline{\varphi} (s) = \varphi (a-s) $ for some $a\in \re$ then $\overline{\sigma} (s) = \sigma (a-s) + \pi$, and the
mean curvature changes sign (changing the orientation of $\varphi$ amounts to changing the unit normal vector to $M$).

The curve $\varphi$ determines a smooth complete embedded hypersurface if and only if it does not intersect
itself, it  is closed and it is orthogonal to the boundary of $ [0, \infty ) \times [0,\pi ] $ at points of intersection.

Except at the points where $x'(s) =0$ we can also express $M$ by a function $p$ defined in some subset of
$[0,\infty )$ with values in $[0,\pi ]$ by the relation $\varphi (s) =(t,p(t))$. In terms of the function $p$, $M$ has mean curvature $h$ if

\begin{equation}p''(x)= (1+p'(x)^2 ) \left( (m-1) \  \frac{\cos p(x)}{\sin p(x)}- (n-1) \ \frac{p'(x)}{x}  -h \sqrt{1+ p'(x)^2} \right) .
\end{equation}

Similarly at points where $y'(s) \neq 0$ one can express $M$ by a function $f$ defined on some subset of $(0, \pi )$ with
values in $[0, \infty )$ by the relation $\varphi (s) = (f(t), t)$. Of course $f$ is the inverse function of $p$.  In terms of the function $p$, $M$ has mean curvature $h$ if

\begin{equation} f''(X) = (1+(f'(x)^2) \left( (m-1) \ \frac{\cos (x)}{\sin (x)} f'(x)  - (n-1) \ \frac{1}{f(x)}  -h \sqrt{1+ h'(x)^2} \right) .
\end{equation}

Note that in (2) we are parmetrizing $\varphi$ so that $x'(s) >0$ while in (3) it is parametrized so that
$y'(s) >0$.

\vspace{.5cm}

We assume from now on that $h$ is a constant and
we will try to find invariant hypersurfaces of mean curvature $h$ by studying the solutions of equation (2).

 Let $x_h \in (0,\pi )$ be the only value such that $h=(m-1) \cot (x_h )$. We have the constant solution $p=x_h$.
The first elementary observation about equation (2) is that if at some $x$ one has $p'(x) =0$ then $x$ is a local
minimum of $p$  if $p(x) <x_h$, $x$ is a local maximum  of $p$ if $p(x)> x_h$ and of course $p$ is the constant solution if
$p(x) = x_h$.

We will need the following completely elementary observation:

\begin{Lemma}
Let $j:[a,b]\rightarrow [m,M]$, $k:[c,d] \rightarrow [m,M]$ be two increasing $C^1$-functions with the same image. 
Denote by $I=k^{-1} \circ j :[a,b] \rightarrow [c,d]$.  Let $x_0 \in (a,b)$ be such that $j'(x_0 ) = k' (I(x_0))$.

If $j'(x_0 ) >0$ we have: 

(a) If $j''(x_0 ) > k'' (I(x_0 ))$ then there exists $\varepsilon >0$ such that $j'(x) < k'(I(x))$ for $x\in (x_0 -\varepsilon ,x_0 )$
and $j'(x) > k'(I(x)) $ for $x\in (x_0 ,x_0 +\varepsilon )$ 


If $j'(x_0 ) =0$ but for $|x - x_0 |>0 $ $j'(x) >0 , k'(I(x)) >0$ we have:

(b) If $j''(x_0 ) = k''(I(x_0 ))<0$ and $j'''(x_0) < k'''(I(x_0 ))$ then  there exists $\varepsilon >0$ such
 that $j'(x) < k'(I(x))$ for $x\in (x_0 -\varepsilon ,x_0 )$.

(c) If $j''(x_0 ) = k''(I(x_0 ))>0$ and $j'''(x_0) > k'''(I(x_0 ))$ then  there exists $\varepsilon >0$ such
 that $j'(x) > k'(I(x))$ for $x\in (x_0 ,x_0 + \varepsilon )$.

\end{Lemma}

\begin{proof} We have $k\circ I = j$ and so

$$k'(I(x)) I'(x) = j'(x)$$

$$k''(I(x) I'(x)^2 + k'(I(x)) I'' (x) = j''(x)$$

$$k'''(I(x)) I'(x)^3 +3k''(I(x)) I'(x) I''(x) + k'(I(x) I'''(x) = j'''(x)$$

When $j'(x_0 ) >0$ it follows from the first equation that $I'(x_0 )=1$. 
In case (a) we have that $I''(x_0 ) >0$. It follows that there exists $\varepsilon >0$ such
that $I'(x) <1$ for $x\in (x_0 -\varepsilon ,x_0 )$ and $I'(x) >1$ for $x\in (x_0 ,x_0 +\varepsilon )$  and (a) follows. 


In case (b) or (c) it follows from the second equation  that $I' (x_0 ) =1$. And then  it follows from the third 
equation that $I'' (x_0 ) >0$. Then (b) and (c) follow as in the case (a).

\end{proof}




The following two lemmas will be the main tool to prove Theorem 1.1 and Theorem 1.2 in the next sections.

\begin{Lemma} Let $p$ be a solution of (2) such that there are  points $x_0  < x_1$ where $x_0$ is a  local maximum
of $p$ and $x_1$ a local minimum, $p$ is strictly decreasing in $[x_0 , x_1 ]$, $p(x_0 )=E, p(x_1 )=e$. 
Then $p$ has a local maximum at a point $x_2 > x_1$ ($p$
strictly increasing in $(x_1 , x_2 )$) and $p(x_2 ) <E$.
\end{Lemma}

\begin{proof}

Let $f(x)= p (2 x_1 -x)$. Then $f'(x) = -p ' (2x_1 -x)$, $f''(x) = p '' (2x_1 -x ) $. It
follows that $f$ verifies

\begin{equation}f''(x)= (1+f'(x)^2 ) \left( (m-1) \ \frac{\cos f(x)}{\sin f(x)}  + (n-1) \ \frac{f'(x)}{2x_1 - x}  -h \sqrt{1+ f'(x)^2} \right)
\end{equation}

\noindent
with $f(x_1 )=e$ and $f' (x_1 )=0$. Also $f'' (x_1 ) = p '' (x_1 ) = (m-1) \cos (e)/ \sin (e) -h >0$. But it is easy to check that
$f''' (x_1 ) - p ''' (x_1 )=2(n-1) f''(x_1) /x_1 >0$. Since both $f$ and $p$ are strictly increasing
after $x_1$ with $f' > p'$ at least close to $x_1$ we have that for any $x> x_1$, $x$ close to $x_1$ there exists a  value $x_p$
close to $x_1$, $x_1 < x < x_p$ such that $f(x) = p(x_p )$. For these values one has that  $f'(x) > p'(x_p )$ by Lemma 2.1 (c).

We know that $f$ is increasing in the interval  $[x_1 , 2 x_1 - x_0 ]$ and $f(2x_1 -x_0 )= E$. Suppose $p$ also increases after
$x_1$ until reaching the value $E$ at some point $x_E$. 
Then for each $x\in (x_1 , 2x_1 - x_0 )$ there exists a unique $x_p \in ( x_1 , x_E )$ such that $f(x) = p (x_p )$. We have seen
that for $x$ close to $x_1$ $f'(x) > p' (x_p )$. Suppose that there exists a first value $x<2x_1 - x_0$ such that $f'(x) = p' (x_p )$.
Looking at equations (2) and (4)  
one has $f''(x) > p '' (x_p )$. But this would imply by Lemma 2.1 (a)  that for $y<x$, $y$ close to $x$ one would have $f'(y) < p' (y_p )$,
contradicting the assumption that $x$ was the first value where the equality holds. It follows that there exist a
first value $z> 2x_1 - x_0$
such that $p'(z)=0$, $p(z) = E$. Then $p''(z) = f''(2x_1 - x_0 ) = (m-1) \cos (E)/ \sin (E) -h <0$ and 
$f'''(2x_1 - x_0 ) - p'''(z) = p''(z) (1/(2x_1 - x_0 ) +1/z) <0$. Then by Lemma 2.1 (b) we would have that for 
some $x<2x_1 - x_0$, $f'(x) <p'(x_p )$ which we saw it cannot happen. It follows that $p$ must reach a local maximum
before reaching the value E, as claimed in the Lemma.

\end{proof}

Similarly one has 

\begin{Lemma}
 Let $p$ be a solution of (2) such that there are  points $x_0  >  x_1$ where $x_0$ is a  local maximum
of $p$ and $x_1$ a local minimum, $p(x_0 )=E, p(x_1 )=e$. Then $p$ has a local minimum at a point $x_2 > x_0$ ($p$
strictly decreasing in $(x_0 , x_2 )$) and $p(x_2 ) >e$.
\end{Lemma}

\begin{proof} Let $f(x)= p (2 x_0 -x)$. Then $f'(x) = -p ' (2x_0 -x)$, $f''(x) = p '' (2x_0 -x ) $. It
follows that $f$ verifies

\begin{equation}f''(x)= (1+f'(x)^2 ) \left( (m-1) \  \frac{\cos f(x)}{\sin f(x)}  + (n-1) \  \frac{f'(x)}{2x_0 - x}  -h \sqrt{1+ f'(x)^2} \right)
\end{equation}

\noindent
with $f(x_0 ) = p(x_0 )$, $f' (x_0 )=0$. Also $f'' (x_0 ) = p '' (x_0 ) = (m-1)  \cos (E)/ \sin (E) -h <0$.

But it is easy to check that
$f''' (x_0 ) - p ''' (x_0 )=2 (n-1)  f''(x_0) /x_0 <0$. Since both $f$ and $p$ are strictly decreasing
after $x_0$ with $f' < p'$ at least close to $x_0$ we have that for any $x> x_0$, $x$ close to $x_0$ there exists a  value $x_p$
close to $x_0$, $x_0 < x < x_p$ such that $f(x) = p(x_p )$. Applying Lemma 2.1 (c) to $-f$ and $-p$ we have that
for these values  $f'(x) < p'(x_p )$.

We know that $f$ is decreasing until $2 x_0 - x_1$ where it reaches its minimum value $e$. 
Suppose $p$ also decreases after
$x_0$ until reaching the value $e$ at some point $x_e$. Then for each $x\in (x_0 , 2x_0 - x_1 )$ there exists a unique $x_p \in ( x_0 x_e )$ such that $f(x) = p (x_p )$. We have seen
that for $x$ close to $x_1$ we have $f'(x) < p' (x_p )$. Suppose that there exists a first value $x\in (x_0 , 2x_0 - x_1 )$ such that $f'(x) = p' (x_p )$.
Looking at equations (2) and (5)  
one has $f''(x) < p '' (x_p )$. 
Applying Lemma 2.1 (a)  (to $-f$ and $-p$)
 this would imply that for $y<x$, $y$ close to $x$ one would have $f'(y) > p' (y_p )$,
contradicting the assumption that $x$ was the first value where the equality holds. It follows that there exist a
first value $z> x_0$
such that $p'(z)=0$, $p(z) =e$.  Then $p''(z) = f''(2x_0 - x_1 ) = (m-1)  \cos (e)/ \sin (e) -h >0$ and 
$f'''(2x_1 - x_0 ) - p'''(z) >0$. Then by Lemma 2.1 (b) we would have that for 
some $x<2x_1 - x_0$ $f'(x) >p'(x_p )$ which we saw it cannot happen. It follows that $p$ must reach a local minimum
before reaching the value e, as claimed in the Lemma.

\end{proof}

\section{Minimal hypersurfaces}

In this section we will prove Theorem 1.1.

We begin with  the following elementary observation which we will need later:

\begin{Lemma} Let $f:[a,b) \rightarrow [c,d]$ be a $C^2$ function, $a\geq 0$, $f'(x) >0 $ for all $x\in (a,b)$, $h(c,d) \rightarrow \re $ be $C^1$ function
such that $h' <0$, $\lim_{x\rightarrow d} h(x).(d-x) \leq -1$. Assume that 

\begin{equation}
f''(x) \leq  (1+f'(x)^2 ) \ h(f(x))
\end{equation}

\noindent
Then $\lim_{x\rightarrow b} f(x) <d$.

\end{Lemma} 

\begin{proof}Assume that  $\lim_{x\rightarrow b} f(x) =d$. Let $x_0 \in (a,b)$ be such that  $h(f(x) ) (d-f(x ))< -1/2$  for all $x \in [x_0 ,d) $. Then
we have that $h(f(x)) < \frac{-1}{2(d-f(x))}$ for $x\in [x_0 , d)$.
Let $\varepsilon = d- f(x_0 ) >0$. Let $r=f'(x_ 0 ) >0$. Let $\delta = \frac{\varepsilon }{2r}$. For $x\in [x_0 , d)$
we have that $f'' (x)<0$ and so $f' (x) < r$. Since we assume that  $\lim_{x\rightarrow b} f(x) =d$ we must have that $r(b-x_0 )> \varepsilon$. Then
$\delta < (1/2) (b-x_0 )$. Moreover,
$f(x_0 + \delta ) < f(x_0 ) + r\delta = d-\varepsilon /2$. Also for $x>x_0$ we have
$f''(x)<\frac{-1}{2(d-f(x))} < \frac{-1}{2\varepsilon}$. And so $f'(x_0 + \delta ) < r -\frac{1}{4r}$. The step in which one goes (less than)  half the distance
between $x_0 $ and $b$ can be repited any number of times. But after doing it a finite times one would get that $f'$ becomes negative, 
contradicting the hypothesis. Therefore  $\lim_{x\rightarrow b} f(x) <d$ as claimed.

\end{proof}

{\it Proof of Theorem 1.1 }:

Assume that the curve $\varphi$ determines a complete immersed connected minimal hypersurface. We write $\varphi (s)=(x(s),y(s))$ and
denote by $\sigma (s)$ the angle function as in the previous section. Then:

\begin{equation} \sigma ' (s) = (m-1) \  \frac{\cos (y(s))}{\sin (y(s))} \cos (\sigma (s)) - (n-1) \  \frac{\sin \sigma (s) }{x(s) }  .
\end{equation}

The first observation is that if $\varphi (s) = (x(s),y(s))$ determines a  minimal hypersurface then so does
$\overline{\varphi} (s) = (x(s), \pi - y(s))$. 

Let ${\bf x_I} = \inf x(s)$. There are three distinct possibilities. 

P1. ${\bf x_I }=0$.

P2. ${\bf x_I }>0$ and there is a point $({\bf x_I} , y)$ in $\varphi$ with $y \neq 0, \pi $.

P3.  ${\bf x_I} >0$   and there is a point $({\bf x_I }, y)$ in $\varphi$ with $y= 0$ or  $y= \pi $.

Apriori P2 and P3 might not exclude each other but we will see that in fact they do.

Consider first the case P1. So we assume
that  the  curve $\varphi$ starts at the $y$-axis, i.e. it contains a point $(0,A)$ with $A\in [0, \pi ]$. By
the previous comments we can assume that $A\in [0,\pi /2 ]$. If $A=\pi /2$ we have the constant
solution $\varphi (s)=(s,  \pi/2 )$, which corresponds to $\sigma =0$ and $M= \re^n \times {\bf S}^{m-1}$. 

When $A=0$ the corresponding hypersurface is not smooth, it has a singularity at the point $V$ which corresponds
to $(0,0)$ in the orbit space (a punctured neighborhood of that point would be diffeomorphic to
$S^{m-1} \times S^{n-1} \times \re$). One can probably study such a singular minimal hypersurface as in
\cite[Theorem 4.1]{Alencar}, but we will not do it here.

Therefore we can assume $A\in (0,\pi /2)$. We then have a curve  $\varphi (s) = (x(s),y(s))$ with $\varphi (0)=
(0,A)$, $y'(0)=0, x'(0)=1$. It follows that $M$ can be described (close to this point at least) by a function $p$
satisfying

\begin{equation}p''(x)= (1+p'(x)^2 ) \left(  (m-1) \ \frac{\cos p(x)}{\sin p(x)}- (n-1) \ \frac{p'(x)}{x}  \right)
\end{equation}

\noindent
with initial conditions $p(0)=A$, $p'(0)=0$ (and therefore $p''(0)= (m-1) \frac{\cos (A) }{\sin (A)}>0$).

The proof of Theorem 1.1 is based in the following proposition:

\begin{Proposition} Let $p$ be a solution of (8), $z\geq 0$ and  $p(z)=A\in (0,\pi )$. Then $p$ is
defined for all $t\in [z,\infty )$, $p(t) \in (0,\pi )$, and $p$ oscillates around $\pi /2$.

\end{Proposition}

\begin{proof} Note that if $p'(x)=0$ then $p$ has a local maximum at $x$ if $p(x) > \pi /2$
and $p$ has a local minimum at $x$ if $p(x)<\pi /2$. We can assume that $p'(z) \geq 0$ and
$p'(x)>0$ for $x>z$ close to $z$ (if not we consider $\overline{p} = \pi - p$).
 We want to show that there exists $x_1 >z$ such that $p'(x_1 )=0$.

Let $[z,x_F)$ be the maximal interval of definition of $p$ and assume that $p'>0$ in this interval. Let
$y_F = \lim_{x \rightarrow x_F} p(x)$. Lemma 3.1 tells us precisely that it cannot happen that $x_F < \infty$ and $y_F =\pi$.
Also if $y_F > \pi /2$ then there is a final interval where $p''$ has a negative upper bound. It would then follow that $x_F < \infty$, and $p$ is increasing
and $p'$ decreasing close to $x_F$; so both have limits and $p$ could be extended beyond $x_F$. Then we must have
$y_F \leq \pi /2$. In the same way, if $x_F < \infty$ and if there is a final interval $(x,x_F )$ where $p''$ does not change sign
then $p$ could be extended beyond $x_F$. But if $p''$ keeps changing signs when $x$ approaches $x_F <\infty$  the 
lengths of the intervals where $p''$ has a fixed sign would approach 0 (as we approach $x_F$). It is easy to see that $p' (x)$  must be bounded 
and then also $p''$ must be bounded; it should then be clear again that it would exist $\lim_{x\rightarrow x_F} p'(x)$ and again $p$ could
be extended beyond $x_F$.  We are left to assume that  $x_F = \infty$ and $y_F \leq \pi /2$.
It is clear that there must be points converging to $\infty$ where $p'' \leq 0$. Assume
that there is a  point $x_0 > \sqrt{\frac{m-1}{n-1}} $ such that $p''(x_0 ) =0$. Then it follows from equation   (8) that 

$$ p'''(x_0 )= (1+p'(x_0 )^2 ) p'(x_0 ) \left( \frac{-(m-1)}{\sin^2 (p(x_0 ))} + \frac{n-1}{x_0^2} \right) <0 .$$

\noindent
It follows that $p''(x) <0$ for all $x>x_0$. Therefore there must exist $x_0 >0$ such that $p''(x) <0$ for all $x>x_0$. This implies that $\lim_{x\rightarrow \infty} p'(x)=0$ and then looking at 
equation (8) one sees that $y_F = \pi /2$. Then we can find $\varepsilon >0$ very small and $x>100 (n-1)$, $x>x_0$ such that $1 + p'(x)^2 <2$, 
$\cos (p(x))/ \sin (p(x)) =\varepsilon$ and then $p'(x) \geq 100 \varepsilon$ (from equation (8), since $p''(x)<0$). Then for $y\in (x,x+1)$, $p''(y) \geq -(1/50) p'(x)$. 
And then $p'(x+1) > p'(x) -(1/50) p'(x) > 50 \varepsilon$. And then $p(x+1) > p(x) + 50 \varepsilon > \pi/2$
(for $\varepsilon$ small enough we have that if $\cos (p(x))/ \sin (p(x)) =\varepsilon$ then $p(x) > \pi /2 - 2\varepsilon$). This is again
a contradiction and it follows that there exists a first value $x_1 >z$ which is a local maximum of $p$.

The same argument 
can now be used to show that there must be a first value $x_2 > x_1$ which is a local minimum of $p$.
Then $p$ will oscillate around $\pi /2$. But moreover it follows from the Lemma 2.2 and Lemma 2.3  that the local
maxima and minima stay bounded away from $\pi$ and 0 (respectively). If the values of $p$ at the  local extrema 
also stay bounded away $\pi /2$ it is elementary and easy to see from equation (8)  
 that the distance between consecutive extrema of $p$  will have
a positive lower bound and  therefore $p$ would be defined for all $x>z$. The only possibility left would be that  
there exists $x_0 >0$ such that $\lim_{x\rightarrow x_0} p(x) = \pi /2$. But then again one would have that  $\lim_{x\rightarrow x_0 } p'(x) = 0$
and then we would have that $p$ must be the constant solution. This finishes the proof of the proposition

\end{proof}

Then coming back to the case P1 we choose a small $z>0$ and apply the proposition to see that the solution $p$ of equation (8) determines a complete
embedded minimal hypersurface (diffeomorphic to  $\re^n \times {\bf S}^{m-1}$).

In the case P3 we can consider for instance the case when there is a point $({\bf x_I} , 0)$ in $\varphi$. Then we choose $z>{\bf x_I}$ close to ${\bf x_I}$ and 
again apply the proposition to see that the corresponding solution $p$ of  equation (8) determines a complete
embedded minimal hypersurface (diffeomorphic to  ${\bf S}^{n-1} \times \re^m$).

Finally in case P2 we can assume we have a point $({\bf x_I},y_0 )$  in $\varphi$ with  $y_0 \in (0,\pi /2 )$. Then  we have two branches of
$\varphi$ coming from the point, each one can be described by a function $p$. One of them will be increasing and the other decreasing after
${\bf x_I}$. For each of the branches we can apply the proposition to see that the corresponding solution $p$ of equation (8)
 is defined in all $({\bf x_I}, \infty)$ and oscillates around $\pi /2$.
It follows that one has a
minimal immersion of $S^{m-1} \times S^{n-1} \times \re$ with self-intersections.

 This completes the proof of Theorem 1.1

\section{$O(n) \times O(m)$-invariant constant mean curvature hypersurfaces}

In this section we will show existence of some invariant constant mean curvature hypersurfaces
by studying the solutions of equation (2), with $h$ a positive constant. Of course the equation is considerably more
complicated than equation (8). Moreover the clear description obtained in Theorem 1.1 will not be possible
in this case. The first observation is that the isoperimetric problem in $\re^n \times S^m$ can be solved and
it is easy to see by standard symmetrization arguments that the hypersurfaces which are the boundaries of
the isoperimetric regions are $O(n) \times O(m)$-invariant. And as usual they have constant mean curvature. 
For small values of the volume the corresponding isoperimetric region will be a ball bounded by constant mean
curvature hypersurface which will be an $O(n) \times O(m)$-invariant sphere. This will be given by a solution
of equation (2), for some value of $h$, with $p(0)=A >0 $ and $p(x)=0$ at some value $x>0$. So we know that
the situation will in general  be different to what happened for the case of minimal hypersurfaces studied
in the previous section. 

As in Section 2 we let $x_h \in (0,\pi /2 )$ be the value such that $(m-1)\cos (x_h )/\sin (x_h ) =h$. The constant function
$p = x_h$ is of course a solution. Consider solutions $p$ of equation (2) with initial 
conditions $p(0)=A$, $p'(0)=0$. We write $p=p(A,x)$. Let 

$$w(x) = \frac{\partial p(A,x)}{\partial A} (x_h ,x).$$ 

\noindent
Then $w$ satisfies $w(0) =1$, $w'(0)=0$ and

\begin{equation} w''(x) = -\frac{m-1}{\sin^2 (x_h )} w  -\frac{n-1}{x} w'.
\end{equation}

The following is easy to check:

\begin{Lemma}
Solutions of equation (9) are oscillating.
\end{Lemma}

\begin{proof}({\it Theorem 1.2})
It follows from the previous lemma that for $A$ close enough to $x_h$ the corresponding solution $p(A,x)$ of
equation (2) must have  a local minimum at a value $x_1 >0$. 
Now we apply Lemma 2.2 to show that there exists $x_2 >x_1$ such that $p$ is increasing in $(x_1 ,x_2 )$ and
$x_2$ is a local maximum of $p$. Then by applying Lemma 2.3 and Lemma 2.2 we see that there exists a sequence
of consecutive local maxima and minima $x_1 < x_2 <x_3 < x_4 ...$  such the sequence of local maxima $p(x_{2i})$ 
is decreasing (and bounded below by $x_h$) and the sequence of local minima $p(x_{2i +1})$ is increasing
(and bounded above by $x_h$). Assume that one of the limits of these monotone sequences is not
$x_h$, for instance $\lim p(x_{2i}) = y >x_h$. If the maximal interval of definition of $p(A,x)$ were a finite interval
$(0,x_f )$ consider the solution $t$ of the equation 

$$t''(x)= (1+t'(x)^2) \left( (m-1) \frac{\cos (t(x))}{\sin (t(x))} -(n-1) \frac{t'(x)}{x_f} - h\sqrt{1+t'(x)^2} \right) $$

\noindent 
with initial conditions $t(0)= y$, $t'(0)= 0$. Let $r>0$ be the first value such that $t(r)=x_h$. The for each $s<r$ for
all $i$ big enough $x_{2i +1} - x_{2i} >s$. This is of course a contradiction and therefore   $p$
would be defined on $(0,\infty )$. If the limit of both monotone sequences is $x_h$ and the maximal
interval of definition of $p$ where a final interval $(0,x_f )$ then we would have that 
$\lim_{x\rightarrow x_f} p(x) = x_f$ and $\lim_{x\rightarrow x_f} p'(x) = 0$: then $p$ must be the constant
solution and we would again reach a contradiction. It follows that $p$ is defined for all $\re_{>0}$ and gives
an embedded hypersurface of constant mean curvature $h$.  

We have proved that there is an open interval containing $x_h$ such that for all $A$ in the interval $p(A,x)$ determines
an embedding of $\re^n \times S^{m-1}$ of constant mean curvature $h$.
Now assume   that $x_h > A>0$ and the corresponding solution $p(A,x)$
does not determine such an embedding. Then from
the previous discussion follows that  $p(A ,x)$ is an increasing function in a maximal interval of definition $(0,x_F )$. 
Let $y_F = \lim_{x \rightarrow x_F} p(A ,x)$.
It follows from Lemma 3.1 that it cannot happen that $x_F < \infty$ and $y_F = \pi$. If $y_F > x_h$ then $p''(A , x) <0$
for $x$ close to $x_F$. It follows that it exists $\lim_{x\rightarrow x_F} p'(A , x)$ and $x_F < \infty$. Then $p(A ,x)$ could be
extended beyond $x_F$, reaching a contradiction. Then we can  assume that $y_F \leq x_h$. If $x_F < \infty$ and there
is a sequence of points $x_i$ approaching $x_F$ where $p'' (x_i ) =0$ one can see from equation (2) that $p'$ must stay bounded.
Then $p''$ must also stay bounded and then   it exists $\lim_{x\rightarrow x_F} p'(A , x)$. This would imply again that $p(A,x)$ could
be extended beyond $x_F$. And the same conclusion can be reached if $p''$ has a constant sign close to $x_F$. It follows
that we must have that $x_F = \infty$ and therefore $p(A,x)$ determines
an embedding of $\re^n \times S^{m-1}$ of constant mean curvature $h$.

Finally let $b<\pi$ and assume that 
for all $a$, $x_h <a<b$ the corresponding solution $p(a,x)$ determines
an embedding of $\re^n \times S^{m-1}$ of constant mean curvature $h$ but this is not true for $p(b ,x)$. Then from
the previous discussion follows that  $p(b ,x)$ is a decreasing function in a maximal interval of definition $(0,x_F )$. 

Let $y_F = \lim_{x \rightarrow x_F} p(b ,x)$.

If $y_F >0$ then it is elementary to see that if $\liminf p'(b,x) > -\infty$ then $p''(b,x)$ is bounded and so 
the limit $\lim_{x\rightarrow x_F} p'(x)$ exists and is finite. It then follows that the solution $p(b,x)$ could be
extended beyond $x_F$. Then we must have that  $\lim_{x\rightarrow x_F} p'(x) = -\infty$. This corresponds to the situation when
in equation (1) $x'(s)=0$. But then one can for instance study the solution by
considering equation (3) instead. The inverse function $f=p^{-1}$ verifies $f'(y_F )=0$ and could be extended to
an interval containing $y_F$. It then follows that for values close to $b$ the corresponding solution of equation (1) 
has the same behavior. This contradicts that for every $a<b$ the solution $p(a,x)$ decreases until reaching  a local minimum. 

It follows that $y_F = 0$. Consider the function

$$q(x) = \frac{p''(x)}{1+(p'(x))^2} =(m-1) \frac{\cos (p(x))}{\sin (p(x))} -(n-1)\frac{p'(x) }{x} -h \sqrt{1+(p'(x)^2} .$$

At a point $x_0$ at which $p''(x_0 ) =0$ we have that 

$$q'(x_0 )= p'(x_0 ) \left( -\frac{m-1}{\sin^2 (p(x_0 ))} + \frac{n-1}{x_0^2} \right).$$

Then for $x_0$ close to $x_F$ we would have that $q'(x_0 ) > 0$. It follows that $q$ and $p''$ where negative before $x_0$
and positive after $x_0$. Therefore $p''$ must have a constant sign close to $x_F$. Then the limit $\lim_{x\rightarrow x_F} p'(x) = L $
exists. If $L$ is finite then it is clear from equation (2) that $p''$ must be positive close to $x_F$. Then there there exist $x_1$ 
close to $x_F$ such that for $x\in (x_1 , x_F )$ we have $p''(x) >1/(2p(x))$.  Since $L$ is finite the speed at which $p$ reaches 0
is bounded. Then the previous inequality would imply that $p'$ must approach $-\infty$. Then  $\lim_{x\rightarrow x_F} p'(x) = -\infty$
and so the solution $p(b,x)$ determines an embedding of $S^{n+m-1}$ of constant mean curvature $h$. 
 
\end{proof}

\section{Stability of the $O(n) \times O(m)$-invariant constant mean curvature hypersurfaces}

We will now consider the stability of the constant mean curvature hypersurfaces described in the past sections. 

We will prove the instability of the $O(n) \times O(m)$-invariant noncompact constant mean curvature  hypersurfaces considered in Theorem 1.1 and Theorem 1.2. 
The arguments for instability go along the lines of the ones given by Pedrosa and Ritor\'e in \cite{Pedrosa}, see also \cite{Xin}.

It is said that the immersion of a hypersurface $j:\Sigma^{k-1}\rightarrow M^k$  with constant mean curvature is stable, if and only if  $Q_{\Sigma}(u)\geq 0$, for all differentiable functions $u:\Sigma^{k-1}\rightarrow \re$, with compact support and such that $\int_{\Sigma} u dA=0$ (see \cite{Barbosa}). With the index form $Q_{\Sigma}(u)$  given by

\begin{equation}
\label{index}
Q_{\Sigma}(u)= \int_{\Sigma} \{ \| \nabla u \|^2  -(Ric(N)+|\textit{B}|^2)u^2\}d{\Sigma},
\end{equation}
\noindent where $N$ is a unit vector normal to  $\Sigma$, $d{\Sigma}$ is the volume element on $\Sigma$, $Ric(N)$, the Ricci curvature of $N$, and $|\textit{B}|$ is the norm of the second fundamental form $\textit{B}$ of $\Sigma$. This is also written as

\begin{equation}
\label{index2}
Q_{\Sigma}(u)= -\int_{\Sigma} u \ Lu \ d{\Sigma},
\end{equation}

\noindent
where $L$ is the Jacobi operator $L(u) = \Delta u +(Ricci (N) + |\textit{B}|^2) u$.

For a hypersurface $ \Sigma \subset \re^n \times {\bf S}^m$ invariant by the  $O(n) \times O(m)$ action and generated by  a curve 
$\varphi (t)$ as in section (2) which solves

\begin{equation}
\label{independent}
\begin{aligned}
 x'(t)&=\cos \sigma (t) \\
 y'(t) &=\sin  \sigma (t)\\
 \sigma '(t)& =(m-1)\cot(y(t)) \cos \sigma (t)-(n-1)\frac{\sin \sigma (t)}{x(t)}- h.  \\
\end{aligned}
\end{equation}

\noindent
With $h$ constant, direct computation gives the following:

\noindent
$Ric(N)=(m-1)\cos^2 \sigma$,

\noindent
$d\Sigma= x^{n-1} \sin^{m-1}(y) d \omega_{m-1} d \omega_{n-1} dt$,  $d \omega_{m}$ is the volume element of the m-sphere,

\noindent
$|\textit{B}|^2= \sigma'^2+(m-1) \cot^2(y) \cos^2(\sigma)+\frac{(n-1)}{x^2} \sin^2(\sigma)$,

\noindent
and $\Delta u=u'' +\left( (n-1)  \frac{x'(t)}{x(t)}+(m-1) \cot(y(t)) y'(t)\right)u'$, for an invariant function $u(t)$.

Hence, we can rewrite the index form for hypersurfaces generated by solutions  of (12) on an invariant 
function $u(t)$, $t\in (t_0 ,t_1 )$, as

$$Q_{\Sigma}(u)= -\omega_{m-1}\omega_{n-1} \int_{t_0}^{t_1} u L u \  x^{n-1} \sin^{m-1}(y) dt,$$
\noindent with 

\begin{equation}
\begin{aligned}
\label{L}
L u&= u'' + \left( \frac{(n-1)}{x}x'(t)+(m-1) \cot(y) y'(t)\right)u'\\
    &+\left((m-1) \cos^2 \sigma+\sigma'^2+(m-1) \cot^2(y) \cos^2(\sigma)+\frac{(n-1)}{x^2} \sin^2(\sigma)\right)u.
\end{aligned}
\end{equation}

There are two canonical examples of invariant  hypersurfaces of constant mean curvature $h$  in $\re^n \times {\bf S}^m$:
the product $\Sigma_h^1 = {\bf S}^{n-1} \times {\bf S}^m$ of a sphere of constant mean curvature $h$ in $\re^n$ with
${\bf S}^m$ and the product $\Sigma_h^2 = \re^n \times {\bf S}^{m-1}$ of $\re^n$ and a hypersphere of mean curvature $h$ in
${\bf S}^m$. In terms of equation (12) they are given by the constant solutions  
 $\varphi (t) =((n-1)/h , t)$ and $\varphi(t)=(t,x_h )$, respectively. We will first consider these two cases: 

\begin{Lemma}
\label{constant}
 For $m,n>1$, the hypersurface $\Sigma_h^2$ given by the constant solution  of equation (12) (with $h$ constant), $\varphi(t)=(t, x_h)$, is unstable.   
\end{Lemma}

\begin{proof}
Let $m,n>1$. An invariant function $u: \Sigma_h^2 \rightarrow \re$ is a radial function on $\re^n$. We have 

$$Q_{\Sigma_h^2 }(u)= \omega_{m-1}\sin^{m-1}  (x_h )   \int_{\re^n}  \| \nabla u \|^2   - k \ u^2  dx $$

\noindent
with 
 $k=\left(\left(\frac{h}{m-1}\right)^2+1\right)(m-1)$.

Then one chooses any $u \neq 0$ with compact support and $\int_{\re^n} u \ =0$. Then for each $\alpha >0$ the
function $u_{\alpha} (x) = u(\alpha x)$ all have mean 0 and one can pick $\alpha$ so that $Q_{\Sigma_h^2}(u_{\alpha} ) <0$
(this is of course just the well known fact that the {\it bottom of the spectrum} of $\re^n$ is 0).

\end{proof}

\begin{Lemma}
\label{constant2}
 For $m,n>1$, the hypersurface $\Sigma_h^1$ given by the constant solution  of equation (12) (with $h$ constant), $\varphi(t)=((n-1)/h,t)$, is unstable if and only if $h>\sqrt{ m(n-1) }$.   
\end{Lemma}

\begin{proof} Note that $\Sigma_h^1 = {\bf S}^{n-1} \times {\bf S}^m$ and an invariant function $u:\Sigma_h^1 \rightarrow \re$  is constant along the $(n-1)$-spheres and
so it can be consider as a function on ${\bf S}^m$.  Then 

$$Q_{\Sigma_h^1 }(u)= \omega_{n-1} \left( \frac{n-1}{h} \right)^{n-1}   \int_{{\bf S}^m}  \| \nabla u \|^2   - \frac{h^2}{n-1} \ u^2  dV_{{\bf S}^m}$$

Hence, the instability condition $Q_{\Sigma}(u)<0$, is equivalent to  
\begin{equation}
\label{quotient2}
\frac{\int_{\S^m} |\nabla u|^2 dV_{\S^m}}{\int_{\S^m}  u^2  d{V_{\S^m}}}< \frac{h^2}{n-1}.
\end{equation}

Thus, by recalling that $m\leq\frac{\int_{\S^m} |\nabla u|^2 dV_{\S^m}}{\int_{\S^m}  u^2  d{V_{\S^m}}}$ for each $u$ with mean 0, since the dimension $m$ is the first Eigenvalue of
(the positive Laplacian on) $S^m$, the instability condition reduces to $m(n-1) < h^2$. 
\end{proof}

We are now ready to prove  Theorem 1.3.

\begin{proof} ({\it Theorem 1.3})

Consider a hypersurface  $\Sigma$ that belongs to one of the families of noncompact constant mean curvature hypersurfaces described in Theorem 1.1 or Theorem 1.2.

Let $f(s)=(x(s),y (s),\sigma (s))$ be the solution of equation (12) that generates $\Sigma$. We have seen in the previous sections  that in all the cases 
considered $\Sigma$ is described by a curve which has at least one end which is given by the graph of a function $p$ satisfying equation (2).  We have seen that 
$p$ has a sequence of maxima and minima as $x \rightarrow \infty$.

Let $\{(x_1,y_1),(x_2,y_2),(x_3,y_3),...\}$ be the set of alternating maxima and minima of $f(s)$, $p(x_i )= y_i $. Consider the function $u=\sin(\sigma)$. Direct computation yields,
$$uLu=(n-1)\frac{y'^2}{x^2},$$
\noindent with $L$ the operator given by eq. (\ref{L}).  Of course,  $u$  can be extended by symmetry to a field on all of $\Sigma$. It follows that 
$$Q_{\Sigma}(u)=-\int_{\Sigma}uLu \ \ d\Sigma<0.$$

\noindent We next note that $u$ vanishes at the set of alternating maxima and minima, and consider
\begin{equation}
\label{u1}
{u}_1(x)= \begin{cases}
   u(x)  & \text{if}\  x \in [x_1 ,x_2 ] \\
  						0    &    otherwise
  						\end{cases} 
\end{equation}
and similarly,
\begin{equation}
\label{u2}
{u}_2(x)= \begin{cases}
   u(x)  & \text{if}\  x \in [x_2 ,x_3 ] \\
  						0    &    otherwise
  						\end{cases} 
\end{equation}

These two functions have disjoint supports and satisfy $Q_{\Sigma}(u_i)<0$, $i=1,2$. It follows that by taking a linear combination of the two, $\bar{u}=C_1u_1+C_2u_2$, we can construct a function such that $\int_{\Sigma}\bar{u}=0$ and $Q_{\Sigma}(\bar{u})<0$.

\end{proof}

\end{document}